\documentclass[a4paper,twoside]{article}

\newcommand{\heuteIst}{March 10, 2008 }

\usepackage{amsmath,amsthm,amsfonts,latexsym,amscd,amssymb,enumerate}
\usepackage{amsrefs}
\usepackage[all]{xy}
\usepackage[hypertex,backref]{hyperref}

\usepackage{color}

\swapnumbers
\theoremstyle{plain}
\newtheorem{theorem}{Theorem}[section]
\newtheorem{lemma}[theorem]{Lemma}
\newtheorem{corollary}[theorem]{Corollary}
\newtheorem{proposition}[theorem]{Proposition}

\theoremstyle{definition}

\theoremstyle{remark}
\newtheorem{remark}[theorem]{Remark}

\usepackage{pb-diagram}

\newcommand{\reals}{\mathbb{R}}
\newcommand{\complexs}{\mathbb{C}}

\newcommand{\integers}{\mathbb{Z}}
\newcommand{\Z}{\mathbb{Z}}
\newcommand{\R}{\mathbb{R}}

\DeclareMathOperator{\id}{id}
\newcommand{\boundary}[1]{\partial#1}
\newcommand{\del}{\partial}

\newcommand{\tensor}{\otimes}

\newcommand{\dlra}[1]{\stackrel{#1}{\longrightarrow}}
\newcommand{\iso}{\cong}

\newcommand{\Pos}{\mathcal{R}^+}

\DeclareMathOperator{\Diff}{Diff}

\DeclareMathOperator{\im}{im}      

\DeclareMathOperator{\Spin}{Spin}

\DeclareMathOperator{\coker}{coker}

\DeclareMathOperator{\ind}{ind}

\newcommand{\forget}[1]{}

{\catcode`@=11\global\let\c@equation=\c@theorem}

\allowdisplaybreaks[2]

\begin{document}
\pagestyle{myheadings}
\markboth{Diarmuid Crowley and Thomas Schick}{The Gromoll filtration and the $\alpha$-invariant}

\date{Last compiled \today; last edited  \heuteIst or later}

\date{\today}

\title{The Gromoll
  filtration, KO-characteristic classes and metrics of
  positive scalar curvature}

\author{Diarmuid Crowley\thanks{
\protect\href{mailto:diarmuidc23@gmail.com}{e-mail:diarmuidc23@gmail.com}
\protect\\
\protect\href{http://www.dcrowley.net}{www:~http://dcrowley.net}}
\\ Max Planck Institute for Mathematics\\ Vivatsgasse 7,
53111 Bonn, Germany \and Thomas Schick\thanks{
\protect\href{mailto:schick@uni-math.gwdg.de}{e-mail:
  schick@uni-math.gwdg.de}
\protect\\
\protect\href{http://www.uni-math.gwdg.de/schick}{www:~http://www.uni-math.gwdg.de/schick}
\protect\\
Partially funded by the Courant Research Center `\! \!``Higher order structures in Mathematics"'
within the German initiative of excellence
}\\
Mathematisches Institut, Georg-August-Universit\"at G{\"o}ttingen\\
Bunsenstr.~3, 37073 G\"ottingen, Germany}
\maketitle

\begin{abstract}
Let $X$ be a closed $m$-dimensional spin manifold which admits a metric 
of positive scalar curvature and let $\Pos(X)$ 
be the space of all
such metrics.  For any $g \in \Pos(X)$, Hitchin used the KO-valued
$\alpha$-invariant to define a homomorphism $ A_{n-1} \colon
\pi_{n-1}(\Pos(X), g) \to KO_{m+n}$. 
%
He then showed that $A_0 \neq 0$ if $m = 8k$ or $8k+1$ and that $A_1 \neq 0$ if $m = 8k-1$ or  $8k$.

In this paper we use Hitchin's methods and extend these results by proving that
\[ A_{8j+1-m} \neq 0 \]
whenever $m \geq 7$ and $8j - m \geq 0$.
The new input are elements with
non-trivial $\alpha$-invariant deep down
in the Gromoll filtration 
of the group $\Gamma^{n+1} = \pi_0(\Diff(D^n, \del))$.
We show that $\alpha(\Gamma^{8j+2}_{8j-5})\ne \{0\}$ for $j\ge 1$.  
This information about elements existing deep in the Gromoll filtration is the
second main new result of this note.

\end{abstract}

\section{Introduction}
Let $n$ be greater than $4$, let $\Theta_{n+1}$ denote the group of homotopy
$(n+1)$-spheres and let $\Gamma^{n+1} = \pi_0({\Diff}(D^n, \del))$ denote the
group of isotopy classes of orientation preserving diffeomorphisms of the
$n$-disc which are the identity near the boundary.  There is the standard
isomorphism
$\Sigma \colon \Gamma^{n+1} \cong \Theta_{n+1}$, due to Smale and Cerf
\cites{Smale,Ce}.  Moreover, for all $0 < i \leq j$ there are homomorphisms
\begin{equation*}
\lambda_{i, j}^n \colon \pi_j({\Diff}(D^{n-j}, \del)) \to
\pi_{j-i}({\Diff}(D^{n-j+i}, \del)) .
\end{equation*}
%
The definitions of $\Sigma$ and $\lambda^n_{i, j}$ are recalled in Section \ref{subsec:smoothingtheory}.

We denote $\lambda : = \lambda^n_{i, i}$.  In \cite{G}*{Abschnitt 1} Gromoll defined the group
%
\[ \Gamma^{n+1}_{i+1} : = \lambda(\pi_i(\Diff(D^{n-i}, \del)) \subset \Gamma^{n+1} \]
%
and the corresponding filtration 
\[ 0 = \Gamma^{n+1}_{n-2} \subset \Gamma^{n+1}_{n-3} \subset \dots \subset \Gamma^{n+1}_3 \subset \Gamma^{n+1}_2 = \Gamma^{n+1} \, . \]
%
We say that $f \in \Gamma^{n+1}$ has Gromoll filtration $i$ if 
$f \in \Gamma^{n+1}_i \setminus \Gamma^{n+1}_{i+1}$. 
The identity $\Gamma^{n+1} = \Gamma^{n+1}_2$ is due to Cerf \cite{Ce}, 
as pointed out in \cite{A-B-K}.  The equality $\Gamma^{n+1}_{n-2} =0$ follows 
from Hatcher's proof \cite{Hatcher} of the Smale Conjecture. 

Starting with Novikov \cite{N}, authors have used the homomorphisms
$\lambda^n_{i, j}$ to explore the homotopy type of $\Diff(D^n, \del)$.  For example, 
\cite{B-L}*{Theorem 7.4} shows that there is an infinite
sequence $\{ (p_i,k_i, m_i) \}$ of integer triples with $p_i$ odd primes, $\lim_{i \to \infty} m_i/k_i = 0$ and
\[ \pi_{k_i} (\Diff(D^{m_i}, \del)) \tensor \Z/p_i \neq 0 \, .\]
Later, Hitchin
\cite{H}*{Section 4.4} used the 
homomorphisms $\lambda^n_{i, j}$ to investigate the homotopy type of the space
of positive scalar curvature metrics on a closed manifold.  In this paper we
extend the results of \cite{B-L} and \cite{H}*{Section 4.4}.


Hitchin's main tool is the $\alpha$-invariant, the KO-valued index of the real
Dirac operator of a closed spin manifold. 
 Since an exotic sphere carries a unique spin structure, we get an induced
homomorphism
\[ \alpha\colon  \Gamma^{m+1}\xrightarrow{\iso}\Theta_{m+1}\to KO_{m+1} \, . \]  
Our first main result shows that the Gromoll filtration of some $(8k+2)$-dimensional exotic spheres with
non-trivial $\alpha$-invariant is quite deep. 

\begin{theorem} \label{thm:1}
For all $j \geq 1$ there is an element $f_j \in \pi_{8j-6}(\Diff(D^7, \del))$
such that $\alpha(\lambda(f_j)) \neq 0$ and $2 f_j = 0$.  Hence
$\alpha(\Gamma^{8j+2}_{8j-5}) \neq \{ 0 \}$ and for all $0 \leq i \leq 8j-6$,
$\lambda^{8j+1}_{i, 8j-6}(f_j) \in \pi_{8j-6-i}(\Diff(D^{7+i}, \del))$ is a
non-trivial element of order $2$.  
\end{theorem}

\subsection{Positive scalar curvature}
Let $X$ be a closed spin manifold of dimension $m$ and let $\Pos(X)$ denote the space of
positive scalar curvature metrics on $X$.  The Lichnerowicz formula entails that the first obstruction to the existence of a positive scalar curvature metric on $X$ is the index of the Dirac operator defined by its spin structure.  This is an element $\ind(X) \in KO_m$ which gives rise to a ring homomorphism
\[ \alpha \colon \Omega_*^{\rm spin} \to KO_*, \quad [X] \mapsto \ind(X) \, .\]
When $X$ is simply connected of dimesion $\ge 5$, Stolz \cite{Stolz} proved
that $\Pos(X) \neq 
\phi$ if and only if $\alpha(X) = 0$.  In general, the question of whether
$\Pos(X) \neq \phi$ is a deep problem which remains open, see for example
\cites{Rosenberg, Sch}. 


If $\Pos(X) \ne \emptyset$ we equip it with the $C^\infty$-topology and go on to
investigate this topological space.  Note that $\Diff(X)$ acts on
$\Pos(X)$ via pull-back of metrics and so fixing $g$
defines a map $T \colon \Diff(X) \to \Pos(X),\; h \mapsto h^* g$.  Moreover, fixing 
$D^m \subset X$ defines an inclusion $i \colon \Diff(D^m, \del) \to \Diff(X)$
via extension by the identity.

Hitchin observed in his thesis \cite{H}*{Theorem 4.7} that sometimes non-zero
elements in $\pi_*(\Diff(D^m, \del))$ yield, via the induced action of $\Diff(D^m, \del)$ on
$\Pos(X)$, non-zero elements in $\pi_*(\Pos(X)) : = \pi_*(\Pos(X), g)$.   More
precisely, Hitchin \cite{H}*{Proposition 4.6} (see Section \ref{subsec:cor1}),
defines a homomorphism
\[ A_{n-1} \colon\pi_{n-1}(\Pos(X)) \to KO_{m+n} \]
and shows that the composition 
\[ C_{n-1} \colon \pi_{n-1}(\Diff(D^m, \del)) \xrightarrow{i_*} \pi_{n-1}(\Diff(X)) \xrightarrow{T_*} \pi_{n-1}(\Pos(X)) \xrightarrow{A_{n-1}} KO_{m+n}    \]
is non-trivial for $n = 1$ and $m = 8k, 8k+1$ and for $n=2$ and $n = 8k-1, 8k$.

Hitchin's method 
exploited the at the time known facts that  $\alpha
(\Gamma^{8j+1}_1)\ne  \{0\}$ and $\alpha(\Gamma^{8j+2}_2)\ne \{0\}$. With our refined knowledge
about the non-zero images $\alpha(\Gamma^{8j+2}_{8j-5})$, we obtain the following corollary using the
same method as Hitchin.

\begin{corollary} \label{cor:1}
Let $X$ be a spin manifold of dimension $m \geq 7$ with $g \in \Pos(X)$ and
let $f_j$ be as in Theorem \ref{thm:1}.  Then for all $j \in \integers$ such
that $8j+1-m \ge 0$, $C_{8j+1-m}(\lambda^{8j+1}_{m-7 ,8j-6}(f_j)) \neq 0 \in
KO_{8j+2}$.  In particular, the homomorphism 
\begin{equation*}
A_{8j+1-m} \colon \pi_{8j+1-m}(\Pos(X)) \to KO_{8j+2}
\end{equation*}
is a split surjection and for all such $(X, g)$ the graded group
$\pi_*(\Pos(X))$ contains non-trivial two-torsion in infinitely many degrees.
\end{corollary}
%
%

To our knowledge, these examples and those of \cite{HSS} are the first examples where
$\pi_k(\Pos(X))$ is shown to be non-trivial when $k > 1$.  In contrast to \cite{HSS}, Corollary 
\ref{cor:1} also shows that $\pi_*(\Pos(X))$ is non-trivial in infinitely many degrees.
However, note that by construction the elements of $\pi_*(\Pos(X))$ found in Corollary
\ref{cor:1}
vanish under the action of $\Diff(X)$, i.e.~in $\pi_*(\Pos(X)/\Diff(X))$. In
contrast to this in
\cite{HSS} the first examples of
elements
$x \in \pi_k(\Pos(X))$ which remain non-trivial by pullback with arbitrary
families in $\Diff(X)$
are constructed for arbitrarily large $k$. That $\Pos(X)/\!\Diff(X) $ often has
infinitely many 
components is already proved in \cites{Botvinnik-Gilkey,LM,PSpsc}.

\section{The Gromoll filtration of Hitchin spheres} \label{sec:GF+HS}

In this Section we prove Theorem \ref{thm:1} and Corollary 
\ref{cor:1}. Section \ref{subsec:smoothingtheory} recalls methods from
smoothing theory which give a second definition of the Gromoll filtration.
Section \ref{subsec:KM} reviews the Kervaire-Milnor analysis of the group of
homotopy spheres.  Section \ref{subsec:alpha} recalls results of Adams from
stable homotopy theory and their relation to the $KO$-index theory due to
Milnor.  Section \ref{subsec:thm1} shows how non-trivial compositions in the
stable homotopy groups of spheres lead to non-zero elements deeper in the
Gromoll filtration and so proves Theorem \ref{thm:1}.  


\subsection{The groups  $\Theta_{n+1}$, $\Gamma^{n+1}$ and $\pi_{n+1}(PL/O)$} \label{subsec:smoothingtheory}
Let $n \geq 5$.  Recall that $\Theta_{n+1}$ is the group of oriented diffeomorphism classes of homotopy $(n+1)$-spheres, that by definition $\Gamma^{n+1} = \pi_0(\Diff(D^n, \del))$ and recall also the space $PL/O$ which will be defined below.  In this subsection we review the three fundamental isomorphisms $\Sigma$, $\Psi$ and $M_*$ appearing the following diagram:
\[ \xymatrix{ \Gamma^{n+1} \ar[rr]^{\Sigma} \ar[dr]^{M_*} & & \Theta_{n+1} \ar[dl]_{\Psi} \\ & \pi_{n+1}(PL/O) } \]
%
%
We then prove that the diagram commutes: a point which seems to have been implicit in the literature.

%
%

Given a mapping class $f \in \Gamma^{n+1}$ we may build a homotopy $(n+1)$-sphere $\Sigma_f$ by first extending $f$ by the identity map to a diffeomorphism $\bar f \colon S^n \to S^n$ and then setting 
$ \Sigma_f : = D^{n+1} \cup_{\bar f} D^{n+1}.$
In this way we obtain the map, which is well known to be a homomorphism
\begin{equation} \label{eq:Sigma}
\Sigma \colon \Gamma^{n+1} \to \Theta_{n+1}, \quad f \mapsto \Sigma_f \, .
\end{equation}
By \cite{Smale} $\Sigma$ is onto and by \cite{Ce} $\Sigma$ is injective.  


Next let $O_k$ and $PL_k$ denote the $k$-dimensional orthogonal group and the
group of piecewise linear homeomorphisms of $k$-dimensional Euclidean space
fixing the origin
and let $O := \lim_{k \to \infty}O_k$ and $PL := \lim_{k \to \infty}PL_k$
denote the corresponding stable groups.  There are inclusions $O_k \to PL_k$
with quotients $PL_k/O_k$ and we obtain the space $PL/O = \lim_{k \to
  \infty}(PL_k/O_k)$ along with stablisiation maps $S \colon PL_k/O_k \to
PL/O$.  The fundamental theorem of smoothing theory applied to the
$(n+1)$-sphere \cites{HM,LR}, (see also \cite{La}*{Theorem 7.3}) states that
there is an isomorphism 
\begin{equation} \label{eq:Gamma}
\Psi_{n+1} \colon \Theta_{n+1} \cong \pi_{n+1}(PL/O) \, .
 %
 \end{equation}
%

A third fundamental result is due to Morlet (unpublished) and 
Burghelea and Lashof \cite{B-L}*{Theorems 4.4, 4.6}.  

\begin{theorem}[\cite{B-L} Theorem 4.4] \label{thm:Morlet}
There is a homotopy equivalence of commutative $H$-spaces
\[ M_n \colon \Diff(D^n, \del) \simeq \Omega^{n+1}(PL_n/O_n) \]
such that the composition
\[ \pi_0\Diff(D^n, \del) \dlra{M_{n*} } \pi_0\Omega^{n+1}(PL_n/O_n) 
\dlra{S_*} \pi_0 \Omega^{n+1}(PL/O) = \pi_{n+1}(PL/O)  \]
%
%
yields an isomorphism
\[ M_* \colon \Gamma^{n+1} \cong \pi_{n+1}(PL/O) \, . \]
Here $S_*$ is induced by the
stabilisation map  $\Omega^{n+1}(PL_n/O_n) \to \Omega^{n+1}(PL/O)$.
%
%
\end{theorem}

To give the alternative description of the Gromoll filtration, we use the
homomorphisms
 
\begin{equation*}
\lambda_{i, j}^n \colon \pi_j({\Diff}(D^{n-j}, \del)) \to
\pi_{j-i}({\Diff}(D^{n-j+i}, \del))
\end{equation*}
from the introduction.  Here we represent $a\in \pi_j(\Diff(D^{n-j},\boundary))$ by a map
\begin{equation*}
a \colon [0,1]^j\to \Diff([0,1]^{n-j},[0,1]^{n-j})
\end{equation*}
such that the value of $a$ is the
identity map near the boundary of $[0,1]^j$ and such that each $a(x)$ is a
diffeomorphism which restricts to the identity near the boundary of
$[0,1]^{n-j}$. The class
$\lambda_{i,j}^n(a)$ is then represented by the map 
\begin{equation} \label{eq:lambda}
\lambda^n_{i,j}(a)\colon [0,1]^{j-i}\to \Diff([0,1]^{n-j}\times [0,1]^i,[0,1]^{n-j}\times[0,1]^i)
\end{equation}
with $\lambda^n_{i, j}(a)(x)(t,y)=(a(x,y)(t),y)$.  Indeed the formula \eqref{eq:lambda}
implies that if we use
$\Omega$ to denote the space of differentiable loops, then there are maps
\[ \Lambda^n_{i, j} \colon \Omega^j\Diff(D^{n-j}, \del) \to \Omega^{j-1}\Diff(D^{n-j+i}, \del) \]
which induce the homomorphisms $\lambda^n_{i, j}$.  

\begin{lemma}[c.f.~\cite{Burghelea_announce}*{Theorem 1.3}]
 \label{lem:Morlet-and-stabilisation}
Let $i_n \colon PL_n/O_n \to PL_{n+1}/O_{n+1}$ be the canonical inclusion and let
$\Omega M_n$ be the map of smooth loop spaces induced by $M_n$
and assume $n \geq 4$.  Then  the following diagram is homotopy commutative.
\[ \xymatrix{ \Omega \Diff(D^n, \del) \ar[d]^{\lambda^n_{1, 1}} \ar[r]^(0.45){\Omega M_n} & \Omega^{n+2}(PL_n/O_n) \ar[d]^{\Omega^{n+2}(i_{n}) } \\
\Diff(D^{n+1}, \del) \ar[r]^(0.4){M_{n+1}} & \Omega^{n+2}(PL_{n+1}/O_{n+1}) } \]
\end{lemma}

\begin{proof}
The corresponding statement for $n \neq 4$ with $PL_n$ replaced by $Top_n$ is given 
in \cite{Burghelea_announce}*{Theorem 1.3} where Burghela considers the map
$h_n \colon \Diff(D^n, \del) \to \Omega^{n+1}(Top_n/O_n)$.  And indeed Burghelea
remarks \cite{Burghelea_announce}*{p.9} that the analagous versions of his 
results  hold when $Top_n$ is replaced by $PL_n$.  

We give a somewhat indirect argument based on the work of Kirby and Siebenmann 
which deduces the commutativity of the diagram above from \cite{Burghelea_announce}*{Theorem 1.3}.
By definition the
map $h_n$ factors through $M_n$ and the canonaical map $\pi_n \colon 
PL_n/O_n \to Top_n/O_n$:
\[ h_N = \pi_n \circ M_n \colon \Diff(D^n, \del) \longrightarrow \Omega^{n+1}(PL_n/O_n) \longrightarrow \Omega^{n+1}(Top_n/O_n).\]
Now there is a fibration sequence
\[  \Omega^{n+1}(PL_n/O_n) \to \Omega^{n+1}(Top_n/O_n) \to \Omega^{n+1}(Top_n/PL_n) \]
and for $n \geq 5$ there is, by \cite{KS}*{Essay V, 5.0 (1)}, a homotopy equivalence
\[ Top_n/PL_n \simeq K(\Z/2, 3). \]  
Hence the space $\Omega^{n+1}(Top_n/PL_n)$ is contractible and the map $\pi_n$
above is a homotopy equivalence.
It follows that the commutativity of Burghelea's diagram \cite{Burghelea_announce}*{Theorem 1.3} entails the commutativity of the diagram above.
\end{proof}

 An immediate consequence of Theorem \ref{thm:Morlet} and \cite{Burghelea_announce}*{Theorem 1.3} is the following alternative definition of the Gromoll filtration.

\begin{corollary} \label{cor:altGromoll}
%
$ \Gamma^{n+1}_{k+1} = M_*^{-1}S_* \bigl( \pi_{n+1}(PL_{n-k}/O_{n-k}) \bigr).$ 
%
\end{corollary}

The following lemma is presumably well known and in particular is implicit in \cite{B-L}.  Since we could not find a reference, we give a proof.

\begin{lemma} \label{lem:M=SigmaPsi}
$M_* = (\Psi \circ \Sigma) \colon \Gamma^{n+1} \xrightarrow{\cong}
\pi_{n+1}(PL/O)$. 
\end{lemma}

\begin{proof}
%

We use the description of $\Psi \colon \Theta_{n+1} \cong \pi_{n+1}(PL/O)$ given 
in \cite{Lueck}*{proof of Theorem 6.48}. 
Given an exotic sphere $\Sigma_f$ obtained from a diffeomorphism
$f \in \Diff(D^n, \del)$, take the PL-homeomorphism $u \colon \Sigma_f \cong
S^{n+1}$ to the standard sphere coming from the Alexander trick. There is an
associated ``derivative'' map between the PL-microbundles of $\Sigma_f$ and
$S^{n+1}$. Using the smooth structures, these PL-bundles are induced from the
smooth tangent bundles which are of course vector bundles. Pulling back with
$u$ to $S^{n+1}$, we then have two $O_{n+1}$-structures on the same
$PL_{n+1}$-bundle over $S^{n+1}$, and the difference of the lifts of structure
group gives a pointed map $S^{n+1}\to PL_{n+1}/O_{n+1}$. By stabilization we
get an 
element of $\pi_{n+1}(PL/O)$, which is by definition $\Psi(\Sigma_f)$.

On the other hand, the map $M_* \colon \pi_0(\Diff(D^{n},\partial)) \to \pi_{n+1}(PL/O)$ 
from \cite{B-L} is defined (after we strip off the technicalities 
associated to the  use of simplicial methods) by first looking at the loop 
$\gamma\colon [0,1]\to PL(D^{n},\partial)$ obtained by applying the
Alexander trick to $f$, with induced path $\bar\gamma\colon [0,1]\to
PL(D^{n},\partial)/\!\Diff(D^{n},\partial)$. The latter corresponds to the inverse in the boundary 
map of the fibration 
\[ PL(D^{n},\partial)\to B\!\Diff(D^{n},\partial)=PL(D^{n},\partial)/\Diff(D^{n},\partial) \, , \]
compare \cite{B-L}*{proof of Theorem 4.2}. The
path of PL-derivatives $t\mapsto D(\gamma_t)$ gives, as above by comparing the
pullbacks of the vector bundle
structure on the PL-microbundle of $D^n$ to the standard vector bundle
structure, a loop of maps  
from $(D^{n},\partial)$ to $PL_{n}/O_{n}$, i.e.~a map $S^{n+1} \to PL_n/O_n$. By
\cite{B-L}*{proof of 4.2 and Section 1}, its stabilization represents
$M_*(\psi) \in \pi_{n+1}(PL/O)$.

Observe that the family of PL-homeomorphisms $D^n\to D^n$ just constructed,
extended by the
identity over a ``second hemisphere'', patch together to the PL-homeomorphism
between the homotopy sphere $\Sigma_f$ and $S^{n+1}$ used in the
definition of $\Psi \circ \Sigma$. Moreover, if we stabilize the family of
differentials by 
the identity of the vertical direction, we obtain the differential of that
PL-homeomomorphism. Finally, the underlying vector bundle structures on the
PL-microbundles patch
together and stabilize to the vector bundle structures on the PL-microbundles
of $\Sigma_f$ and
$S^{n+1}$ encountered above.
It follows that the stable comparison maps $S^{n+1} \to PL/O$ coincide,
i.e.~$M_* = \Psi \circ \Sigma$.
\end{proof}

\begin{remark}
  It is interesting to observe that
$\Psi\circ\Sigma$ factors by construction through
$\pi_{n+1}(PL_{n+1}/O_{n+1})$, whereas $M_*$ even factors through $\pi_{n+1}(PL_n/O_n)$.
\end{remark}


\subsection{Homotopy spheres} \label{subsec:KM}
In this subsection we review a number of important isomorphisms used to study
the group of homotopy spheres $\Theta_{n+1}$.  More information and proofs
can be found in \cite{Lueck}*{6.6} and \cite{Le}*{Appendix}.  Let $G :=
\lim_{k \to 
  \infty}G(k)$ denote the stable group of homotopy self-equivalences of
spheres, let $\pi_i^S$ denote the $i$th stable stem and let $\Omega_i^{\rm
  fr}$ denote $i$-dimensional framed bordism group.  We have isomorphisms
\[ \pi_i(G) \cong \pi_i^S \cong \Omega_i^{\rm fr}  \]
where the first isomorphism may be found in \cite{MM}*{Corollary 3.8} and the
second is the Pontrjagin-Thom isomorphism. 

The canonical map $O \to G$ induces the stable $J$-homomorphism on homotopy
groups $J_i \colon \pi_i(O) \to \pi_i(G)$.
%
The group ${\im}(J_i) \subset \pi_i(G)$ is a cyclic summand 
and the group $\coker(J_i)$ maps isomorphically onto the torsion subgroup of $\pi_i(G/O)$ under the canonical map $q \colon G \to G/O$.  Moreover there is an isomorphism $\pi_i(G/O) \cong \Omega^{\rm alm}_i $ where $\Omega^{\rm alm}_*$ denotes almost framed bordism (cycles are manifolds with a chosen base point and a framing of the stable normal bundle on the complement of this base point).  

\begin{theorem}[\cite{KM}*{Section 4}] \label{thm:KM}
For $n \geq 4$ the abelian group $\Theta_{n+1}$ is finite and lies in an exact
sequence 
\[ 0 \dlra{} bP_{n+2} \dlra{} \Theta_{n+1} \dlra{\Phi} \coker(J_{n+1}) \]
where $bP_{n+2}$ is the finite cyclic subgroup of homotopy spheres bounding
parallelizable manifolds. By  \cite{KM}*{Theorem 6.6}, $\Phi$ is surjective if
$n$ is odd. 
\end{theorem}

\begin{proposition}
\label{prop:x}
  The canonical map $p \colon PL/O \to G/O$ satisfies
  \begin{equation*}
q_* \circ \Phi = p_*
  \circ \Psi \colon \Theta_{n+1} \to \pi_{n+1}(G/O).
\end{equation*}

\end{proposition}
\begin{proof}
  The statement follows from the commutativity of the squares
  \begin{equation*}
    \begin{CD}
      &&\pi_{n+1}(PL/O) @<{\Psi}<{\iso}< \Theta_{n+1}\\
      && @VV{p_*}V @VVV\\
      &&\pi_{n+1}(G/O) @<{\iso}<< \Omega^{\rm alm}_{n+1}\\
     &&  @AA{q_*}A @AAA\\
      \pi_{n+1}^S @>{\iso}>> \pi_{n+1}(G) @<{\iso}<< \Omega^{\rm fr}_{n+1}
    \end{CD}
  \end{equation*}
  which is explained in \cite{Lueck}*{Theorem 6.48}.  The homomorphism
  $\Phi$ is geometrically defined as the composition of the upper right
  homomorphism, the isomorphism $\Omega^{\rm alm}_{n+1} \cong \pi_{n+1}(G/O)$ and the
  inverse of the isomorphism induced by $q_*$ from $\coker(J_{n+1})$ to the torsion 
  subgroup of $\pi_{n+1}(G/O)$. 
\end{proof}

\subsection{The $\alpha$-invariant} \label{subsec:alpha}
Recall from \cite{H}*[Section 4.2] that the $\alpha$-invariant 
is the ring homomorphism 
$\alpha \colon \Omega_*^{\Spin} \to KO_{*}$
which associates to a spin bordism class the KO-valued index of the Dirac
operator of a 
representative spin manifold.  We also write $\alpha$ for the corresponding
invariant on framed bordism: 
\begin{equation} \label{eq:alpha}
 \alpha \colon \Omega_*^{\rm fr} \to \Omega_*^{\Spin} \to KO_{*} \, . 
 \end{equation}
Under the Pontrjagin-Thom isomorphism $\Omega_*^{\rm fr} \cong \pi_*^S$ the
$\alpha$-invariant has the following interpretation as Adams' $d$-invariant 
\cite{A}*{Section 7}, $d_\R \colon \pi_*^S \to KO_*$, 
which was used already in \cite{H}*{p.~44}, compare \cite{M}*{Section 3}. 

\begin{lemma} \label{lem:alpha=d}
Under the Pontryagin-Thom isomorphism $\Omega_*^{\rm fr}\iso
\pi_*^S$ the $\alpha$-in\-va\-riant $\alpha \colon \Omega_{8j+1}^{\rm fr} \to
KO_{8j+1}$ 
may be identified with $d_\R \colon \pi_{8j+1}^S \to
KO_{8j+1}$ .  
\end{lemma}


Recall that $KO_*$ satisfies Bott periodicity of period $8$ with Bott
generator $\beta\in KO_8\cong\integers$.   By \cite{A}*{Theorems 7.18 and
  12.13}, for all $k \geq 1$ there are (not uniquely defined) Adams' elements
$\mu_{8k+1} \in \pi_{8k+1}^S = \Omega^{\rm fr}_{8k+1}$ satisfying
\[ \alpha(\mu_{8k+1}) = \alpha(\eta) \beta^k \neq 0 \in KO_{8k+1} \, , \]
where $\eta \in \pi_1^S$ generates the $1$-stem and $\alpha(\eta)$ generates
$KO_{1}$.  Since $\alpha$ is a ring homomorphism we see that $\alpha(\eta
\mu_{8k+1}) = \alpha(\eta^2) \beta^k \neq 0 \in KO_{8k+2}$, and combining Lemma
\ref{lem:alpha=d} with \cite[Proposition 12.14]{A} we have 
%
\begin{equation} \label{eq:alpha_mult}
\alpha (\mu_{8j+1} \cdot \mu_{8k+1} ) = \alpha(\eta^2) \beta^{j+k} \neq 0 \in KO_{8(j+k) + 2} \,. 
\end{equation}



Recall that an element $x\in \pi_j^S=\lim_k \pi_{j+k}(S^k)$ is said to
\emph{live on $S^k$} if there is $x_k\in
\pi_{j+k}(S^k)$ which maps
to $x$ under the canonical homomorphism.

The next crucial property of the elements $\mu_{8k+1}$ is that (at least if we
make suitable choices here) they all
live on $S^4$.

\begin{lemma} \label{lem:Curtis}
For suitable choices, the (not uniquely defined) homotopy class $\mu_{8j+1} \in
\pi_{8j+1}^S$ lives on the $5$-sphere and moreover there
is $\mu_{8j+1, 5}\in \pi_{8j+6}(S^5)$ with $2 \mu_{8j+1, 5} = 0$. It follows
that there is a corresponding homotopy class $\mu_{8j+1, 9}\in \pi_{8j+10}(S^9)$ of
order $2$.
\end{lemma}

\begin{proof}
The  statement follows by carefully inspecting Adams'
construction of the homotopy class $\mu_{8j+1} \in \pi_{8j+1}^S$, involving
Toda brackets.

Let us recall that, given homotopy classes of maps $u\colon S^a\to S^b$,
$v\colon S^b\to S^c$ and $w\colon S^c\to S^d$ such that $[v\circ u]=0$ and $[w\circ
v]=0$, there is a set $\{w,v,u\}$ of homotopy
classes of maps $S^{a+1}\to S^c$, the Toda brackets of $w,v,u$, a kind of
secondary composition. The elements
of the set depend on choices of null-homotopies for $v\circ u$ and $w\circ v$,
and indeed (for $a\ge 1$) $\{w,v,u\}$ is a coset of $[E u]\circ
\pi_{b+1}(S^c)+\pi_{a+1}(S^b)\circ [w] \in \pi_{a+1}(S^c)$, where $E$ denotes
 suspension.

Now, for the construction of the $\mu_{8j+1,5}$ on starts with a homotopy
class $\alpha_1\colon S^{k+7}\to S^k$ of order 
$2$ such that $\{2,\alpha,2\}$ contains $0$. Here $2$ stands for the self map
of the sphere of degree $2$.

One then chooses inductively $\alpha_s\colon S^{k+8s-1}\to S^k$ to be any
element in the Toda bracket $\{\alpha_{s-1},2,\alpha\}$. For notational
simplicity we write $\alpha$ also instead of the 
appropriate suspension of it. Note that in this proof we follow Adams and use $`\alpha_s'$
to refer to a certain homotopy class. This should not be confused with the
$\alpha$-invariant of \eqref{eq:alpha}.

For the induction to work we have to show that
$[2\alpha_s]=0\in\pi_{k+8s-1}(S^k)$. For this we use \cite{Toda}*{Proposition
  1.2 IV)}: $\{\alpha_{s-1},2,\alpha\}2=\alpha_{s-1}\circ \{2,\alpha,2\} =0$.
The latter follows because by our induction hypothesis $[\alpha_{s-1}\circ 2]=0$ and
$\{2,\alpha,2\}$ contains by assumption only multiples of $2$.

Finally, we define $\mu_{8j+1,k-1}$ as any element in the Toda bracket
$\{\eta_{,k-1},2,\alpha_j\}$. Here, we let $\eta_{,n}\colon S^{n+1}\to S^n$ represent (for $n\ge
3$) the generator of $\pi_{n+1}(S^n)\iso \integers/2$.

To see that $\mu_{8j+1,k-1}$ is of order $2$ we need some preparation:

 If for $a\in \pi_{k+s}(S^k)$ we have that $\{2,a,2\}=
  2\pi_{k+s+1}(S^k)\subset \pi_{k+s+1}(S^k)$, then for arbitrary $x\colon
  S^r\to S^{k+s}$ and $y\colon S^k\to S^b$ also $\{2,a,2x\}\subset
  2\pi_{r+1}(S^k)$ and $\{2y,a,2\}\subset 2\pi_{k+s+1}(S^b)$. Note that
  $\{2,a,2x\}$ is a coset of $2\pi_{r+1}(S^k)+\pi_{k+s+1}(S^r)\circ 2
  Ex\subset 2\pi_{r+1}(S^k)$ so it suffices to show that $0\in \{2,a,2x\}$,
  and similarly for $\{2y,a,2\}$. Now the module property
\cite{Toda}*{Proposition 1.2 IV} implies
  $0=0 \circ x \in \{2,a,2\}\circ x\subset \{2,a,2x\}$, and in the same way
  $0\in \{2,a,2y\}$.

 Now
we show by induction that $\{2,\alpha_s,2\}$ consists of the multiples of
$2$. By assumption this is true for $s=1$. For the induction, we apply the
Leibniz rule \cite{Toda}*{Proposition 1.5} which says that
$\{2,\alpha_s,2\}=\{2,\{\alpha_{s-1},2,\alpha\},2\}$ (which is a coset of the
multiples of $2$) is congruent to the set
\[ \{\{2,\alpha_{s-1},2\},\alpha,2\} + \{2,\alpha_{s-1},\{2,\alpha,2\}\}. \]
By the induction hypothesis and the above consideration, both these iterated Toda
brackets only contain multiples of $2$, and so $\{2,\alpha_s,2\}$ must be the
coset of $0$ of the multiples of $2$.

Finally, using again \cite{Toda}*{Proposition 1.2}
\begin{equation*}
 2\mu_{8j+1,k-1}\in  \{\eta_{,k-1},2,\alpha_j\} 2 = \eta_{,k-1}\circ
 \{2,\alpha_{j},2\} \subset \eta_{,k-1} \circ 2\pi_{k+8j}(S^k) =0
\end{equation*}
because $2\eta_{,k-1}=0$ as long as $k\ge 4$.

Finally, we follow literally one of the proofs Adams gives to show that
$\alpha(\mu_{8j+1})$ is non-trivial. This uses the fact, estabilished in
\cite{A}*{p.~68} that for the relevant dimension $\alpha$ coincides with Adams'
homomorphism $e_\complexs$ (both considered to be maps to $\reals/\integers$).
To compute $e_\complexs(\mu_{8j+1})$ one can inductively apply 
\cite{A}*{Theorem 11.1}. This theorem states that
$e_\complexs\{x,2,y\}=2e_\complexs(x)e_\complexs(y)$ modulo $\integers$. Finally,
one only has to use that $e_\complexs(\eta)=1/2$ and
$e_\complexs(\alpha)=1/2$, which is established in \cite{A}*{proof of Theorem 12.13}.

For the choice of $\alpha_1$ we follow again \cite{A}*{proof of Theorem 12.13}
which uses corresponding results of Toda.
Indeed, in
\cite{Toda}*{Lemma 5.13} Toda checks that the element $\sigma'''\in
\pi_{5+7}(S^5)$ of
order $2$ stabilizes to the element of order $2$ in $\pi_7^S$. Moreover, with $E$
still denoting
the suspension, Toda shows in \cite{Toda}*{Corollary 3.7} that 
$\{2,E\sigma''',2\} \ni E\sigma'''\eta_{,13} = 2\sigma''\eta_{,13}=0$
since $\eta_{,13}$ has order $2$. Therefore, an appropriate choice is
$\alpha_1:=E(\sigma''')=2\sigma''\in \pi_{6+7}(S^6)$.  Here
$\sigma''$ is Toda's notation for an element of order $4$ in $(\pi_{6+7}(S^6)
\cong \integers/60\integers$.
\end{proof}

\begin{remark}
  On the face of it, our construction of $\alpha_s$ and therefore $\mu_{8j+1}$
  is slightly more general than Adams' construction which does seem not allow for
  arbitrary elements in the Toda brackets involved in the inductive
  construction. Note, however that we have to use unstable Toda brackets,
  which means that the same construction, starting with larger $k$, might give
  rise to more elements in $\pi_{8j+1}^s$ which do not live on the
  $5$-sphere. 
\end{remark}

\begin{remark} \label{rem:Curtis}
Another proof of the existence of $\mu_{8j+1, 5}$ comes from \cite{Cu}
where Curtis
calculated the sphere of origin for many examples using the the Adams spectral
sequence
and the restricted lower central series spectral sequence.  In fact Curtis
shows that elements of non-trivial $d$-invariant
live on $S^3$.  We gave an independent proof to avoid the task
of checking how the notations from \cite{Cu} match with those of \cite{A}
and to show that there is a $\mu_{8j+1,5}$ of order two.
\end{remark}
%

\subsection{Proof of Theorem \ref{thm:1}} \label{subsec:thm1}

In this subsection we prove our main theorem.  Since every homotopy sphere
has a unique spin-structure we obtain the $\alpha$-invariant on $\Gamma^{n+1}
\cong \Theta_{n+1}$: 
\[ \alpha \colon \Gamma^{n+1} \to \Omega_{n+1}^{\Spin} \to KO_{n+1} \,. \]
Combining \cite{M}*{Theorem 2 and its proof}, \cite{A}*{Theorems 7.18 and
  12.13} and \ref{thm:KM} we see that for each $j>1$ there is a
homotopy $8j-7$-sphere $\Sigma_{\mu_{8j-7}}\in\Theta_{8j-7}$ representing
$[\mu_{8j-7}]\in \coker(J_{8j-7})$.  In particular we have the equation 
$\alpha(\Sigma_{\mu_{8j-7}}) = \alpha(\eta)\beta^{j-1} \ne 0 \in KO_{8j-7}$.
By Cerf's theorem \cite{Ce}, $\Gamma^9_2 = \Gamma^9_1$ and so we can find $g \in \pi_1(\Diff(D^7, \del))$ 
such that $\Sigma(\lambda(g)) = \Sigma_{\mu_9}$. By \eqref{eq:alpha_mult} above, 
\begin{equation}\label{eq:comput_alpha}
\alpha ( \Sigma^{}_{\mu_9} \times \Sigma^{}_{\mu_{8j-7}}) = \alpha(\eta^2)
\beta^j \ne 
0 \in KO_{8j+2} \, .
\end{equation}

Recall the homotopy equivalence $M \colon \Diff(D^7, \del) \simeq \Omega^8(PL_7/O_7)$ of Theorem \ref{thm:Morlet} and consider the induced isomorphism 
%
\[ M_{7*} \colon \pi_1(\Diff(D^7, \del)) \cong \pi_9(PL_7/O_7) \, . \]
With $g \in \pi_1(\Diff(D^7, \del))$ as above we have $M_{7*}(g) \in \pi_9(PL_7/O_7)$.
Now let $\mu_{8j-7, 9} \in \pi_{8j+2}(S^9)$ be an element of order $2$ with $S(\mu_{8j-7, 9}) = \mu_{8j-7}\in\pi_{8j-7}^S$ whose existence is proven in Lemma \ref{lem:Curtis}.  The composition 
\[ M_{7*}(g) \circ \mu_{8j-1, 9} \in \pi_{8j+2}(PL_7/O_7) \]
%
has order $2$ and we define
\[ f_j := M_{7*}^{-1}(M_{7*}(g) \circ \mu_{8j-7, 9}) \in \pi_{8j-6}(\Diff(D^7, \del)) \]
so that $\lambda(f_j) \in \Gamma^{8j+2}_{8j-5}$.  For $\Sigma_{f_j} : = \Sigma(\lambda(f_j))$ we show below that 
\begin{equation} \label{eq:calc_alpha}
 \alpha(\Sigma^{}_{f_j}) = \alpha ( \Sigma^{}_{\mu_9} \times \Sigma^{}_{\mu_{8j-7}}) 
\end{equation}
and so by \eqref{eq:comput_alpha} we have that $\alpha(\lambda(f_j)) = \alpha(\Sigma^{}_{f_j})  = \alpha(\eta^2) \beta^j \neq 0 \in KO_{8j+2}$ which proves Theorem \ref{thm:1}.   

We prove equation \eqref{eq:calc_alpha} using the following diagram where $k=8j+2$.  We obtain the diagram by combining \cite{B-L}*{p.\,14} and \cite{Lueck}*{Theorems 6.47, 6.48} and we claim that it commutes:

\begin{equation}\label{eq:big_diagram}
  \begin{CD}
    \pi_1(\Diff(D^7,\partial))\times \pi_k(S^9)  && \pi_{k-8}(\Diff(D^7,\partial)) @>{\Sigma\circ\lambda}>> \Theta_k\\
      @VV{M_*\times \id}V @V{\iso}V{M_*}V  @VV{=}V\\
    \pi_9(PL_7/O_7)\times \pi_{k}(S^9) @>{\circ}>>
    \pi_{k}(PL_7/O_7) @>{\Psi^{-1}\circ S_*}>> \Theta_k\\
    @VV{S_* \times \id}V  @VV{S}V @VV{=}V\\
    \pi_9(PL/O)\times \pi_{k}(S^9) @>{\circ}>> \pi_{k}(PL/O) @<{\Psi}<{\iso}< \Theta_{k}\\
    @VV{p_*\times S}V @VV{p_*}V @VVV\\
    \pi_9(G/O)\times \pi_{k-9}^S @>{\circ}>> \pi_{k}(G/O) @<{\iso}<<
    \Omega^{\rm alm}_k @>{\alpha}>> KO_k\\ 
    @AA{q_*\times \id}A @AA{q_*}A @AAA @AA=A\\
    \pi_9(G)\times \pi_{k-9}^S @>{\circ}>> \pi_{k}(G) @<{\iso}<< \Omega^{\rm
      fr}_k @>{\alpha}>> KO_k\\   
    @AA{\iso}A @AA{\iso}A @AA{=}A @AA{=}A\\
    \pi_9^S\times \pi_{k-9}^S @>{\circ}>> \pi_{k}^S @>{\iso}>> \Omega^{\rm
        fr}_k @>{\alpha}>> KO_{k}
  \end{CD}
\end{equation}

Using the claimed commutativity of diagram \eqref{eq:big_diagram} 
let us start in the second row with the pair
\[ (M_{7*}(g), \mu_{8j-7, 9} ) \in \pi_9(PL_7/O_7)\times \pi_{8j+2}(S^9). \]
Since $\Sigma(\lambda(g)) = \Sigma_{\mu_9}$, the pair $(\mu_9,\mu_{8j-7}) \in \pi_9^S \times \pi_{k-9}^S$
maps to the same element in $\pi_9(G/O)\times \pi_{k-9}^S$
as $(M_{7*}(g), \mu_{8j-7, 9} )$.
We already checked in Equation
\eqref{eq:comput_alpha} that $(\mu_9,\mu_{8j-7}) $ is mapped in the bottom row to $\alpha(\eta^2)\beta^j\in KO_{8j+2}$.
Finally, $\Sigma_{f_j}$ is obtained from the element $\Sigma \circ \lambda
\circ M_{7*}^{-1}(M_{7*}(g) \circ \mu_{8j-7, 9}) \in \Theta_{8k+2}$ in the top right corner of the
diagram. By commutativity, its $\alpha$-invariant is as desired.

Now we prove the commutativity of \eqref{eq:big_diagram}. The left part is
taken from \cite{B-L}, the identification of the
homotopy groups of $PL/O$, $G/O$, $G$ with the bordism groups or $\Theta_k$
and the corresponding commutativity
from \cite{Lueck}*{Section 6}. The only assertions which are not contained in
those two
references are the compatibility with $\alpha$, which is clear, and, although
implicitly stated in \cite{B-L}, the commutativity of 
the diagram
\begin{equation*}
  \begin{CD}
\pi_{k-8}(\Diff(D^7,\partial)) @>\Sigma \circ \lambda >> \Theta_k\\
     @V{M_*}V{\iso}V  @VV{=}V\\
        \pi_{k}(PL_7/O_7) @>{\Psi^{-1}\circ S_*}>> \Theta_k    ,
  \end{CD}
\end{equation*}
This commutativity we have essentially prove in Lemma \ref{lem:M=SigmaPsi},
one has additionally only to apply compatibility of the constructions with suspension.

\begin{remark}
The argument above started from the statement $\Sigma^{-1}(\Sigma^{}_{\mu_9})
\in \Gamma^9_2$. If
one knew that a 9-dimensional Hitchin sphere $\Sigma{\mu_9}$ had Gromoll
filtration $\Gamma^9_k$ for $2 < k \leq 5$ then we could repeat the argument
to conclude that $\alpha(\Gamma^{8j+2}_{8j-7+k}) \neq 0$.  As of writing, it
seems that nothing is known about the Gromoll filtration of $9$-dimension
Hitchin spheres beyond the Cerf-Hatcher bounds
$\Sigma^{-1}(\Sigma^{}_{\mu_9})\in \Gamma^9_2$ and $\Gamma^9_6=\{0\}$. 
\end{remark}

\begin{remark}\label{rem:Toda}
  In our construction, we crucially use the ring structure of $KO_*$ and the
  non-triviality of the product of generators in $KO_{8k+1}$. This means that
  the interesting elements (with non-trivial $\alpha$-invariant) we obtain are
  in $\pi_k(\Diff(D^n,\boundary))$ with $k+n\equiv 1\pmod 8$. 

 We expect that one can use Toda brackets (of an element in
 $\pi_*(PL_k/O_k)$ with elements of $\pi_*(S^n)$) to construct such elements
 in $\pi_k(\Diff(D^n,\boundary))$ with $k+n \not\equiv 1 \pmod 8$. This we leave for
 future work.
\end{remark}

\subsection{Positive scalar curvature metrics: Corollary \ref{cor:1} } \label{subsec:cor1}
To prove Corollary \ref{cor:1} one need only recall the arguments following \cite[Proposition 4.6]{H}:
Let $X$ be a closed  $m$-dimensional spin-manifold ($m\ge 7$) and let
$\Pos(X)$ be the space of positive  scalar curvature metrics on $X$ which we assume to be non-empty.  
Observe that the group of diffeomorphisms
of $X$, $\Diff(X)$, acts on $\Pos(X)$ by composition.  In particular, fixing a
metric $g \in\Pos(X)$, define the map 
\[ T \colon\Diff(X) \to \Pos(X), \quad h \mapsto h^*g \, .   \]
Moreover, by fixing a $k$-disc $D^m \subset X$ and extending diffeomorphisms by the identity we obtain a map $i \colon\Diff(D^m, \del) \to \Diff(X)$.

In \cite[Proposition 4.6]{H} Hitchin defines a homomorphism
\[ A_{n-1} \colon\pi_{n-1}(\Pos(X)) \to KO_{m+n} \, . \]
He shows then that the composite homomorphism
\[ B_{n-1} := A_{n-1} \circ T_* \colon\pi_{n-1}(\Diff(X)) \to \pi_{n-1}(\Pos(X), g_0) \to KO_{m+n} \]
assigns to $\phi\colon S^{n-1}\to\Diff(X)$ the family index of the bundle of
spin manifolds $X\to Z_\phi\to S^n$ obtained by the usual clutching
construction. Moreover, in \cite{H}*{Section 4.3, in particular Proposition
  4.4} Hitchin shows that if we start with $\phi\colon S^{n-1}\to
\Diff(D^m,\partial)$ then $B(i_*(\phi))=\alpha(\Sigma_\phi)$, where
$\Sigma_\phi$ is the exotic $(n+m)$-sphere defined by $\lambda(\phi) \in \Gamma^{n+m}_n$.

Fix $j$ with $8j+1 > m \geq 7$.  We apply the argument above starting from $f_j$ as in Theorem \ref{thm:1} and
$\phi := \lambda^{8j+1}_{m-7, 8j-6}(f_j) \in
\pi_{8j+1-m}(\Diff(D^m,\partial))$.  By Theorem \ref{thm:1} we have that $2
\phi = 0$ and that $\lambda(\phi) \in \Gamma^{8j+2}_{8j-5}$ satisfies
$\alpha(\lambda(\phi)) \neq 0$.  Pulling back the metric $g$ by $\phi$ we
obtain a continuous family of metrics in $\Pos(X)$ parameterized by
$S^{8j+1-m}$ and hence the homotopy class $T_*i_*(\phi) \in
\pi_{8j+1-m}(\Pos(X))$ of order $2$. By \cite{H}*{Proposition 4.4},
$A_{8j+1-m}(T_*i_*(\phi)) = \alpha(\lambda)$ and so generates $KO_{8j+2} \cong
\Z/2$.  This proves Corollary \ref{cor:1}.



\begin{appendix}
  \section{The Gromoll filtration: table of values} \label{sec:Gromoll}

  We think that our results about the Gromoll filtration and the existence
  of elements rather deep down with non-trivial $\alpha$-invariant are
  interesting in their own right.  In this appendix we place them in context by
  assembling some results from the literature about the Gromoll filtration.
  
\medskip
  \begin{tabular}{|l|p{7.5cm}|}
 \hline
  $\Gamma_2^7 \cong \Z/{28}$ & $\Gamma_2^7 \neq \Gamma_3^7 \supset 0 =
  \Gamma_4^7$. The inequality for $\Gamma_3^7 \neq \Gamma_2^7$ is due to Weiss
    \cite{W2} who proved that $\Gamma_3^7$ has at most $14$ elements.\\ \hline
  $\Gamma_2^8 \cong \Z/2$ & nothing known 
  \\ \hline
  $\Gamma_2^9 \cong (\Z/2)^3$ & 
  \\ \hline
  $\Gamma_{2}^{10} \cong \Z/6$ & $\Gamma_{3}^{10} \supset \Z/2$ \emph{by Theorem
  \ref{thm:1}}  
  \\ \hline
  $\Gamma_{2}^{11} \cong \Z/{992}$ & $\Gamma^{11}_3 \subset \Z/{496}$ by \cite{W1}
  \\ \hline
  $\Gamma_{2}^{12} = 0$ &\\ \hline
  $\Gamma_{2}^{13} \cong \Z/3$ & $\Gamma_{2}^{13} = \Gamma_{3}^{13} =
  \Gamma_{4}^{13}$ by \cite{A-B-K} 
  \\ \hline
  $\Gamma_{2}^{14} \cong \Z/2$ & nothing known\\ \hline
  $\Gamma_{2}^{15} \cong \Z/2 \oplus \Z/{8,\!128}$ & $\Gamma_{3}^{15} \cong \Z/2
  \oplus \Z/{4,\!064}$ by \cites{A-B-K, W1}\\ \hline
  $\Gamma_{2}^{16} \cong \Z/2$ &nothing known,  conjecturally $\Gamma_{3}^{16}
  = 0$\\ \hline
   $\Gamma_{2}^{17} \cong (\Z/2)^2$ & If Remark \ref{rem:Toda} could be
   implemented we would be able to conclude that  $\alpha(\Gamma^{17}_{9})\ne
   0$ or perhaps even $\alpha(\Gamma^{17}_{10})\ne 0$, in particular
   $\Gamma^{17}_{9}$ or even $\Gamma^{17}_{10}$ would contain $\integers/2$.\\
   \hline 
   $\Gamma_{2}^{18} \cong \Z/8 \oplus \Z/2$ & \emph{By Theorem \ref{thm:1}},
   $\alpha(\Gamma^{18}_{11})\ne 0$. Because $\integers/8=\ker(\alpha)$,
   $\Gamma^{18}_{11} \supset \Z/2$.\\ \hline
 \end{tabular}

\end{appendix}


\begin{bibdiv}
\begin{biblist}
		

\bib{A}{article}{
   author={Adams, J. F.},
   title={On the groups $J(X)$. IV},
   journal={Topology},
   volume={5},
   date={1966},
   pages={21--71},
   issn={0040-9383},
   review={\MR{0198470 (33 \#6628)}},
}
	
		

\bib{A-B-K}{article}{
   author={Antonelli, P.},
   author={Burghelea, D.},
   author={Kahn, P. J.},
   title={Gromoll groups, ${\rm Diff}S\,^{n}$ and bilinear constructions
   of exotic spheres},
   journal={Bull. Amer. Math. Soc.},
   volume={76},
   date={1970},
   pages={772--777},
   issn={0002-9904},
   review={\MR{0283809 (44 \#1039)}},
}
		

 \bib{Botvinnik-Gilkey}{article}{
    author={Botvinnik, Boris},
    author={Gilkey, Peter B.},
    title={The eta invariant and metrics of positive scalar curvature},
    journal={Math. Ann.},
    volume={302},
    date={1995},
    number={3},
    pages={507--517},
    issn={0025-5831},
    review={\MR{1339924 (96f:58159)}},
    doi={10.1007/BF01444505},
 }

	
	
\bib{Burghelea_announce}{article}{
   author={Burghelea, Dan},
   title={On the homotopy type of ${\rm diff}(M^{n})$ and connected
   problems},
   language={English, with French summary},
   note={Colloque International sur l'Analyse et la Topologie
   Diff\'erentielle (Colloq. Internat. CNRS, No. 210, Strasbourg, 1972)},
   journal={Ann. Inst. Fourier (Grenoble)},
   volume={23},
   date={1973},
   number={2},
   pages={3--17},
   issn={0373-0956},
   review={\MR{0380840 (52 \#1737)}},
}
			

 \bib{B-L}{article}{
   author={Burghelea, Dan},
   author={Lashof, Richard},
   title={The homotopy type of the space of diffeomorphisms. I, II},
   journal={Trans. Amer. Math. Soc.},
   volume={196},
   date={1974},
   pages={1--36; ibid. 196\ (1974), 37--50},
   issn={0002-9947},
   review={\MR{0356103 (50 \#8574)}},
}

\bib{Ce}{article}{
   author={Cerf, Jean},
   title={La stratification naturelle des espaces de fonctions
   diff\'erentiables r\'eelles et le th\'eor\`eme de la pseudo-isotopie},
   language={French},
   journal={Inst. Hautes \'Etudes Sci. Publ. Math.},
   number={39},
   date={1970},
   pages={5--173},
   issn={0073-8301},
   review={\MR{0292089 (45 \#1176)}},
}


\bib{Cu}{article}{
   author={Curtis, Edward B.},
   title={Some nonzero homotopy groups of spheres},
   journal={Bull. Amer. Math. Soc.},
   volume={75},
   date={1969},
   pages={541--544},
   issn={0002-9904},
   review={\MR{0245007 (39 \#6320)}},
}

  \bib{G}{article}{
   author={Gromoll, Detlef},
   title={Differenzierbare Strukturen und Metriken positiver Kr\"ummung auf
   Sph\"aren},
   language={German},
   journal={Math. Ann.},
   volume={164},
   date={1966},
   pages={353--371},
   issn={0025-5831},
   review={\MR{0196754 (33 \#4940)}},
}
\bib{HSS}{unpublished}{
  author={Hanke, Bernhard},
  author={Schick, Thomas},
  author={Steimle, Wolfgang},
  title={The space of positive scalar curvature metrics},
  note={in preparation},
}
		

\bib{Hatcher}{article}{
   author={Hatcher, Allen E.},
   title={A proof of the Smale conjecture, ${\rm Diff}(S^{3})\simeq {\rm
   O}(4)$},
   journal={Ann. of Math. (2)},
   volume={117},
   date={1983},
   number={3},
   pages={553--607},
   issn={0003-486X},
   review={\MR{701256 (85c:57008)}},
   doi={10.2307/2007035},
}

\bib{HM}{book}{
   author={Hirsch, Morris W.},
   author={Mazur, Barry},
   title={Smoothings of piecewise linear manifolds},
   note={Annals of Mathematics Studies, No. 80},
   publisher={Princeton University Press},
   place={Princeton, N. J.},
   date={1974},
   pages={ix+134},
   review={\MR{0415630 (54 \#3711)}},
}

\bib{H}{article}{
   author={Hitchin, Nigel},
   title={Harmonic spinors},
   journal={Advances in Math.},
   volume={14},
   date={1974},
   pages={1--55},
   issn={0001-8708},
   review={\MR{0358873 (50 \#11332)}},
}

\bib{KM}{article}{
   author={Kervaire, Michel A.},
   author={Milnor, John W.},
   title={Groups of homotopy spheres. I},
   journal={Ann. of Math. (2)},
   volume={77},
   date={1963},
   pages={504--537},
   issn={0003-486X},
   review={\MR{0148075 (26 \#5584)}},
}

\bib{KS}{book}{
   author={Kirby, Robion C.},
   author={Siebenmann, Laurence C.},
   title={Foundational essays on topological manifolds, smoothings, and
   triangulations},
   note={With notes by John Milnor and Michael Atiyah;
   Annals of Mathematics Studies, No. 88},
   publisher={Princeton University Press},
   place={Princeton, N.J.},
   date={1977},
   pages={vii+355},
   review={\MR{0645390 (58 \#31082)}},
}



\bib{La}{article}{
   author={Lance, Timothy},
   title={Differentiable structures on manifolds},
   conference={
      title={Surveys on surgery theory, Vol. 1},
   },
   book={
      series={Ann. of Math. Stud.},
      volume={145},
      publisher={Princeton Univ. Press},
      place={Princeton, NJ},
   },
   date={2000},
   pages={73--104},
   review={\MR{1747531 (2001d:57035)}},
}


\bib{Le}{article}{
   author={Levine, J. P.},
   title={Lectures on groups of homotopy spheres},
   conference={
      title={Algebraic and geometric topology},
      address={New Brunswick, N.J.},
      date={1983},
   },
   book={
      series={Lecture Notes in Math.},
      volume={1126},
      publisher={Springer},
      place={Berlin},
   },
   date={1985},
   pages={62--95},
   review={\MR{802786 (87i:57031)}},
   doi={10.1007/BFb0074439},
}

\bib{LR}{article}{
   author={Lashof, R.},
   author={Rothenberg, M.},
   title={Microbundles and smoothing},
   journal={Topology},
   volume={3},
   date={1965},
   pages={357--388},
   issn={0040-9383},
   review={\MR{0176480 (31 \#752)}},
}


\bib{LM}{book}{
   author={Lawson, H. Blaine, Jr.},
   author={Michelsohn, Marie-Louise},
   title={Spin geometry},
   series={Princeton Mathematical Series},
   volume={38},
   publisher={Princeton University Press},
   place={Princeton, NJ},
   date={1989},
   pages={xii+427},
   isbn={0-691-08542-0},
   review={\MR{1031992 (91g:53001)}},
}
	
\bib{Lueck}{article}{
   author={L{\"u}ck, Wolfgang},
   title={A basic introduction to surgery theory},
   conference={
      title={Topology of high-dimensional manifolds, No. 1, 2},
      address={Trieste},
      date={2001},
   },
   book={
      series={ICTP Lect. Notes},
      volume={9},
      publisher={Abdus Salam Int. Cent. Theoret. Phys., Trieste},
   },
   date={2002},
   pages={1--224},
   review={\MR{1937016 (2004a:57041)}},
}

\bib{MM}{book}{
   author={Madsen, Ib},
   author={Milgram, R. James},
   title={The classifying spaces for surgery and cobordism of manifolds},
   series={Annals of Mathematics Studies},
   volume={92},
   publisher={Princeton University Press},
   place={Princeton, N.J.},
   date={1979},
   pages={xii+279},
   isbn={0-691-08225-1},
   review={\MR{548575 (81b:57014)}},
}

	
\bib{M}{article}{
   author={Milnor, John W.},
   title={Remarks concerning spin manifolds},
   conference={
      title={Differential and Combinatorial Topology (A Symposium in Honor
      of Marston Morse)},
   },
   book={
      publisher={Princeton Univ. Press},
      place={Princeton, N.J.},
   },
   date={1965},
   pages={55--62},
   review={\MR{0180978 (31 \#5208)}},
}

\bib{N}{article}{
   author={Novikov, S. P.},
   title={Differentiable sphere bundles},
   language={Russian},
   journal={Izv. Akad. Nauk SSSR Ser. Mat.},
   volume={29},
   date={1965},
   pages={71--96},
   issn={0373-2436},
   review={\MR{0174059 (30 \#4266)}},
}

\bib{PSpsc}{article}{
   author={Piazza, Paolo},
   author={Schick, Thomas},
   title={Groups with torsion, bordism and rho invariants},
   journal={Pacific J. Math.},
   volume={232},
   date={2007},
   number={2},
   pages={355--378},
   issn={0030-8730},
   review={\MR{2366359 (2008j:58035)}},
   doi={10.2140/pjm.2007.232.355},
}



\bib{Rosenberg}{article}{
   author={Rosenberg, Jonathan},
   title={$C^{\ast} $-algebras, positive scalar curvature, and the Novikov
   conjecture},
   journal={Inst. Hautes \'Etudes Sci. Publ. Math.},
   number={58},
   date={1983},
   pages={197--212 (1984)},
   issn={0073-8301},
   review={\MR{720934 (85g:58083)}},
}
\bib{Sch}{article}{
   author={Schick, Thomas},
   title={A counterexample to the (unstable) Gromov-Lawson-Rosenberg
   conjecture},
   journal={Topology},
   volume={37},
   date={1998},
   number={6},
   pages={1165--1168},
   issn={0040-9383},
   review={\MR{1632971 (99j:53049)}},
   doi={10.1016/S0040-9383(97)00082-7},
}

\bib{Smale}{article}{
   author={Smale, Stephen},
   title={Generalized Poincar\'e's conjecture in dimensions greater than
   four},
   journal={Ann. of Math. (2)},
   volume={74},
   date={1961},
   pages={391--406},
   issn={0003-486X},
   review={\MR{0137124 (25 \#580)}},
}


\bib{Stolz}{article}{
   author={Stolz, Stephan},
   title={Simply connected manifolds of positive scalar curvature},
   journal={Ann. of Math. (2)},
   volume={136},
   date={1992},
   number={3},
   pages={511--540},
   issn={0003-486X},
   review={\MR{1189863 (93i:57033)}},
   doi={10.2307/2946598},
}
\bib{Toda}{book}{
   author={Toda, Hirosi},
   title={Composition methods in homotopy groups of spheres},
   series={Annals of Mathematics Studies, No. 49},
   publisher={Princeton University Press},
   place={Princeton, N.J.},
   date={1962},
   pages={v+193},
   review={\MR{0143217 (26 \#777)}},
}
\bib{W1}{article}{
   author={Weiss, Michael},
   title={Sph\`eres exotiques et l'espace de Whitehead},
   language={French, with English summary},
   journal={C. R. Acad. Sci. Paris S\'er. I Math.},
   volume={303},
   date={1986},
   number={17},
   pages={885--888},
   issn={0249-6291},
   review={\MR{870913 (87m:57038)}},
}	
\bib{W2}{article}{
   author={Weiss, Michael},
   title={Pinching and concordance theory},
   journal={J. Differential Geom.},
   volume={38},
   date={1993},
   number={2},
   pages={387--416},
   issn={0022-040X},
   review={\MR{1237489 (95a:53057)}},
}
		
\end{biblist}
\end{bibdiv}
\end{document}